\documentclass[journal,twoside,web]{ieeecolor}
\usepackage{generic}
\usepackage{cite}
\usepackage{amsmath,amssymb,amsfonts}

\usepackage{amsthm}
\usepackage{bm}
\usepackage{graphicx}
\usepackage{subcaption}
\usepackage{hyperref}
\hypersetup{hidelinks}
\usepackage[font=footnotesize]{caption}
\usepackage{booktabs}
\usepackage{textcomp}
\let\labelindent\relax
\usepackage{enumitem}
\usepackage{float}
\usepackage{stfloats}
\usepackage{algpseudocodex}
\usepackage{algorithm}
\def\BibTeX{{\rm B\kern-.05em{\sc i\kern-.025em b}\kern-.08em
    T\kern-.1667em\lower.7ex\hbox{E}\kern-.125emX}}
\title{On the Optimality of Markovian Policies for Chance-Constrained Covariance Steering}
\author{Naoya Kumagai, \IEEEmembership{Graduate Student Member,~IEEE}, Kenshiro Oguri, \IEEEmembership{Member,~IEEE}
\thanks{This material is based upon work supported by the Air Force Office of Scientific Research under award number FA9550-23-10512. The authors are with the School of Aeronautics and Astronautics, Purdue University, West Lafayette, IN 47907 USA. Emails: \{nkumagai, koguri\}@purdue.edu.}}

\newcommand{\R}{\mathbb{R}}

\newcommand{\E}{\mathbb{E}}
\newcommand{\Symmetric}{\mathbb{S}}

\newcommand{\cov}[1]{\mathrm{Cov}(#1)}
\let\P\relax
\newcommand{\P}[1]{\mathbb{P}(#1)}
\let\E\relax
\newcommand{\E}[1]{\mathbb{E}[#1]}

\newcommand{\normal}{\mathcal{N}}

\newtheorem{theorem}{Theorem}

\newtheorem{lemma}{Lemma}
\newtheorem{proposition}{Proposition}

\newtheorem{definition}{Definition}

\theoremstyle{definition} 
\newtheorem{remark}{Remark}

\usepackage{cleveref}
\crefname{align}{}{}
\crefname{equation}{}{}
\crefname{figure}{Fig.}{Figs.}
\crefname{table}{Table}{Tables}
\crefname{theorem}{Theorem}{Theorems}
\crefname{definition}{Definition}{Definitions}
\crefname{lemma}{Lemma}{Lemmas}
\crefname{remark}{Remark}{Remarks}
\crefname{assumption}{Assumption}{Assumptions}
\crefname{proof}{Proof}{Proofs}
\crefname{algorithm}{Algorithm}{Algorithms}
\crefname{problem}{Problem}{Problems}
\crefname{proposition}{Proposition}{Propositions}
\crefname{corollary}{Corollary}{Corollaries}
\crefname{section}{Section}{Sections}
\crefname{appendix}{Appendix}{Appendices}

\allowdisplaybreaks 

\def\showChanges{0} 

\newcommand{\highlight}[1]{%
    \ifnum\showChanges=1
        \textcolor{blue}{#1}%
    \else
        #1%
    \fi
}

\newcommand{\todohighlight}[1]{%
    \ifnum\showChanges=1
        \textcolor{red}{#1}%
    \else
        #1%
    \fi
}

\newcommand{\mathhl}[1]{
    \ifnum\showChanges=1
        \mathcolor{blue}{#1}
    \else
        #1
    \fi
}

\usepackage{marginnote}
\usepackage{tcolorbox}
\newcommand{\margincomment}[1]{%
    \ifnum\showChanges=1
    \textcolor{blue}{\setlength{\baselineskip}{10pt}\footnotesize\raggedright\marginnote{#1}}%
    \else
    \fi
}

\NewDocumentEnvironment{subalign}{b}{
    \begin{subequations}
    \allowdisplaybreaks        
    \begin{align}
        #1
    \end{align}
    \end{subequations}
}{}

\NewDocumentEnvironment{subalignsmall}{b}{
    \begin{subequations}
    \small                     
    \allowdisplaybreaks        
    \begin{align}
        #1
    \end{align}
    \end{subequations}
}{}

\begin{document}
\maketitle

\begin{abstract}
	Many studies on finite-horizon stochastic optimal control, including covariance steering, parameterize control policies as state-history-affine. This parameterization enables a convex reformulation, thereby yielding a tractable solution method.
	However, the necessity of dependence on previous states has not been well established. \textit{Is this dependence necessary, or merely an artifact of the convex reformulation?} We show that it is an artifact that can be removed losslessly.
	Given an optimal solution of the state-history-affine formulation, we construct a deterministic Markovian policy which is affine in the current state. 
	We show that, even for the covariance steering problem with a broad class of commonly used state and control safety constraints, the synthesized Markovian policy almost surely produces the same control actions as the history-dependent policy and therefore the same state trajectories, cost, and moments. 
	Thus, every optimum of the history-dependent formulation admits a lossless Markovian transformation. Geometrically, the history-dependent formulation lifts the policy space for convexity, and its optimal solution can be projected back to the Markovian policy space.
	We extend the analysis to output feedback and a convex upper-bounding surrogate for value-at-risk costs.
\end{abstract}


\graphicspath{{../figures/}}

\section{Introduction}
\IEEEPARstart{C}{ovariance} steering has emerged as a promising approach for stochastic optimal control problems, where the goal is to steer the state of a stochastic system to a desired distribution within a finite time horizon. 
Its ability to \textit{design closed-loop uncertainty evolution}, rather than \textit{evaluate it after control synthesis}, has led to many applications requiring decision-making under uncertainty, such as path planning for autonomous vehicles \cite{okamotoOptimalStochasticVehicle2019} and spacecraft \cite{oguriChanceConstrainedControlSafe2024a,kumagaiRobustCislunarLowThrust2025}, as well as dynamic portfolio optimization \cite{skafDesignAffineControllers2010}.
The conic reformulations commonly used in covariance steering naturally allow the addition of probabilistic (chance) constraints on the state and control, which can specify safety regions and actuator limits \cite{charnesDeterministicEquivalentsOptimizing1963,vanhessemConicReformulationModel2002,okamotoOptimalCovarianceControl2018,okamotoOptimalStochasticVehicle2019}. While covariance control has been studied in both continuous- and discrete-time settings \cite{hotzCovarianceControlTheory1987,chenOptimalSteeringLinear2018,okamotoOptimalCovarianceControl2018,balciConstrainedMinimumVariance2024,liuOptimalCovarianceSteering2025}, we focus on the discrete-time setting in this work.

Recent solution methods for discrete-time finite-horizon covariance steering involve reformulating the problem as a convex optimization problem. Formulations in \cite{skafDesignAffineControllers2010,okamotoOptimalCovarianceControl2018,balciConstrainedMinimumVariance2024} use state-history-affine or disturbance-history-affine policies. These formulations can express chance constraints in convex form, \textit{at the expense of solving over history-affine policies}; when solving over this large space, both the number of decision variables and the online memory requirements scale quadratically in the horizon length. 
Offline, this results in large linear matrix inequalities (LMIs), which can cause memory issues and slow computation for the solver \cite{rapakouliasDiscreteTimeOptimalCovariance2023}; chance constraints further compound the issue by adding more LMIs \cite{kumagaiSquareRootFactorizedCovariance2026}.
Online, the resulting controller must store gains and state histories for all previous time steps; thus, this memory-dependent policy may be undesirable or even prohibitive for resource-constrained applications \cite{foxMinimuminformationLQGControl2016,changMemoryAwareEmbeddedControl2017}.
Alternative formulations in \cite{liuOptimalCovarianceSteering2025,rapakouliasDiscreteTimeOptimalCovariance2023} instead use Markovian policies that are affine in the current state. This approach results in deterministic reformulations with smaller LMIs and faster computation. However, it has several disadvantages, including nonconvex chance constraints (which necessitate sequential convex programming) and limited flexibility in the choice of cost function because the formulation is tightly connected to the lossless convexification method \cite{kumagaiSquareRootFactorizedCovariance2026}.
\cite{liuOptimalCovarianceSteering2025} establishes the optimality of this Markovian policy for covariance steering without chance constraints, which we refer to as \textit{unconstrained} covariance steering. Under chance constraints, however, the relationship between the state-history-affine and Markovian formulations has not been established.

Outside of covariance steering, several works in finite-horizon stochastic control provide insights into the policy structure that are relevant here. 
Disturbance-history-affine and state-history-affine feedback are shown to be equivalent in linear systems for robust control \cite{goulartOptimizationStateFeedback2006a} and covariance steering \cite{balciConstrainedMinimumVariance2024}.
\cite{zhangRelationshipOptimalState2023} compares the optimal controller for an infinite-horizon LQR problem with that of a finite-horizon disturbance-history-affine LQR problem. They show that as the horizon length of the latter increases, the gain corresponding to the most recent disturbance converges to the optimal infinite-horizon LQR gain.
Disturbance-affine policies are shown to be optimal for one-dimensional constrained robust optimization \cite{bertsimasOptimalityAffinePolicies2010}.
Classic \cite{dermanNoteMemorylessRules1966} and modern \cite{feinbergSufficiencyMarkovPolicies2022} results in Markov Decision Processes (MDPs) show that a history-dependent randomized policy can be replaced by a \textit{randomized} Markov policy. The existence of such randomized policies has been shown \cite{liuOptimalCovarianceSteering2025} and used \cite{balciExactSDPFormulation2022} for covariance steering specifically. However, these results do not provide a deterministic Markov policy. A deterministic policy is desirable in domains that emphasize reproducibility, explainability, and safety, such as aerospace.

This work addresses the existing gap in covariance steering theory: the optimality of deterministic Markovian policies for \textit{chance-constrained} covariance steering.
We show that for such problems where the constraints obey a mild assumption satisfied by most practical safety constraints, it is possible to synthesize a deterministic Markovian policy that performs equivalently to the optimal state-history-affine policy.
These constraints, which we refer to as being covariance-monotonic, approach the feasibility boundary as the covariance of the uncertain variable increases; a precise definition is given in \cref{def:covariance-monotonic}. 
The two policies (history-affine and Markovian) produce the same control actions almost surely and, consequently, the same state trajectories.
Moreover, in the special case in which the initial-state and process-noise covariances are strictly positive definite, we show that the gains corresponding to the state at previous time steps are zero at the optimum,
and the optimal solution is Markovian, even under (covariance-monotonic) chance constraints.

The results have both theoretical significance and practical benefits. 
On the theoretical side, our results extend the optimality of the Markovian structure for covariance steering \cite{liuOptimalCovarianceSteering2025} to chance-constrained problems.
The results support the use of history-feedback parameterization for solving chance-constrained covariance steering while retaining the desirable Markovian structure. 
We uncover that these policies are lifted from the (more restrictive) Markovian policy space, and the optimal solution either lies in, or admits an equivalent projection onto, the Markovian policy space.

Turning to practical benefits, the methods of \cite{okamotoOptimalCovarianceControl2018,balciConstrainedMinimumVariance2024} return an optimal state-history-affine policy, which requires quadratic storage. 
The synthesized Markovian policy enables linear storage, and previous gains and states can be discarded. 
Because the proposed post-processing occurs after optimization, it does not affect the convex offline formulation of the chance-constrained problem.

We generalize the results to an output-feedback (Kalman-filter-in-the-loop) scenario and extend them to a convex upper-bounding surrogate of value-at-risk costs. Both require nontrivial analysis and are important for practical applications.
Numerical investigations on examples from the covariance steering literature support our theoretical findings.

\textbf{Summary of Contributions:}
\begin{itemize}
	\item We show that for chance-constrained covariance steering problems with a mild assumption on the chance constraints, the optimal state-history feedback policy can be replaced by a deterministic Markovian state feedback policy that produces the same state and control trajectories almost surely.
	\item We show that when the initial state and disturbance covariances are positive definite, the optimal state-history feedback policy is readily Markovian without post-processing.
	\item We extend the results to output feedback and value-at-risk costs.
\end{itemize}

\textit{Notation:}
$\R$, $\R^n$, $\R^{n\times m}$, and $\Symmetric_+^n$ denote the set of real numbers, $n$-dimensional real vectors, $n \times m$ real matrices, and symmetric positive semidefinite matrices of size $n \times n$, respectively.
$\mathrm{Tr}(\cdot)$, $\E{\cdot}$, $\cov{\cdot}$, $\P{\cdot}$, and $(\cdot)^\dagger$ denote the trace of a matrix, expectation, covariance, probability, and the Moore-Penrose pseudoinverse, respectively.
$X \succeq 0$ denotes that $X$ is positive semidefinite and $X \succ 0$ denotes that $X$ is strictly positive definite.
At times, we use the notation $x_{i:j}$ to denote $(x_i, x_{i+1}, \dots, x_j)$.

\section{Problem Description} \label{sec:description}
Consider the synthesis of a policy $\pi$ for the optimal control of a discrete-time linear system with additive stochastic disturbances, subject to constraints on the state and control.
\begin{equation} \label{eq:problem}
	\begin{split}
	\min_\pi \quad &J = \sum_{k=0}^{N} \E{x_k^\top Q_k x_k} + \sum_{k=0}^{N-1} \E{u_k^\top R_k u_k}, \\
	\text{subject to} \quad &x_{k+1} = A_k x_k + B_k u_k + G_k w_k, \\
	&g_x(\E{x_k}, \cov{x_k}) \leq 0, \quad k = 1, \dots, N, \\
	&g_u(\E{u_k}, \cov{u_k}) \leq 0, \quad k = 0, \dots, N-1
	\end{split}
\end{equation}
where \( x_k \in \mathbb{R}^n \) is the state, \( u_k \in \mathbb{R}^m \) is the control input, and \( w_k \in \mathbb{R}^l \) is zero-mean white Gaussian noise with unit covariance. We assume that the initial state is normally distributed with known mean and covariance, $x_0 \sim \mathcal{N}(\mu_0, P_0)$, where $P_0 \succeq 0$, and that $x_0$ is independent of $\{w_k\}_{k=0}^{N-1}$.
The weighting matrices satisfy $Q_k \succeq 0$ and $R_k \succ 0$ at every time step.

The constraints $g_x$ and $g_u$ in \cref{eq:problem} are assumed to be covariance-monotonic, as defined below:
\begin{definition}[Covariance-monotonic constraints] \label{def:covariance-monotonic}
	For a constraint $g(\mu, P) \leq 0$ with $\mu \in \R^n$ and $P \in \Symmetric_+^n$, 
\begin{equation}
	P_1 \preceq P_2 \implies g(\mu, P_1) \leq g(\mu, P_2) 
\end{equation}
\end{definition}
The physical intuition behind this definition is that an increase in the covariance of the state or control pushes the solution toward infeasibility.
This definition includes deterministic reformulations (in terms of the moments) of affine \cite{okamotoOptimalCovarianceControl2018} and Euclidean-norm \cite{oguriChanceConstrainedControlSafe2024a} chance constraints on Gaussian random variables, and covariance inequalities for covariance steering \cite{okamotoOptimalCovarianceControl2018}. It also includes mean equality constraints.
Affine and Euclidean-norm chance constraints can be reformulated as 
\begin{align}
	&\P{a^\top z \leq b} \geq 1-\epsilon \ \Leftrightarrow \ a^\top \mu +  \Phi_{\normal}^{-1}(1-\epsilon) \sqrt{a^\top P a} \leq b, \label{eq:affine-cc}\\
	&\P{\|z\|_2 \leq b} \geq 1-\epsilon \ \Leftarrow \ \|\mu\|_2 + \alpha_{l, \epsilon} \sqrt{\lambda_{\max}(P)} \leq b \label{eq:cc-norm},
\end{align}
for $ z \in \R^{l} $ with $ z \sim \normal(\mu, P) $,
where $\Phi_{\normal}^{-1}$ is the inverse cumulative distribution function of a standard normal random variable. Here, $\alpha_{l, \epsilon}:=\sqrt{\Phi_{\chi_l^2}^{-1}(1-\epsilon)}$, where $\Phi_{\chi_l^2}^{-1}$ denotes the inverse cumulative distribution function of a chi-square random variable with $l$ degrees of freedom. Affine chance constraints of \cref{eq:affine-cc} are covariance-monotonic when $\epsilon \in (0, 0.5]$.

\subsection{Affine Controllers}
Our goal is to construct a Markovian controller, parameterized as
\begin{equation} \label{eq:markov-controller}
	\pi_k^M(x_k) = v_k + H_k (x_k - \E{x_k}).
\end{equation}
For some special cases of \cref{eq:problem}, such as unconstrained covariance steering, such controllers are known to be optimal \cite{liuOptimalCovarianceSteering2025}. $v_k$ is the feedforward term and $H_k$ is the feedback gain term, both to be synthesized. $\E{x_k}$ is the mean of the state at time step $k$, which will also be known via optimization.

Also consider a state-history-affine controller:
\begin{equation*}
		u_k = \pi_k^C(x_0, x_1, \dots, x_k).
\end{equation*}
Without loss of generality, this controller is parameterized as 
\begin{equation} \label{eq:history_controller}
	u_k = v_k + \sum_{i=0}^{k} K_{k,i} (x_i - \E{x_i}). 
\end{equation}
Compared with the Markovian case, a larger number of variables, \(\{K_{k,i}:0\le i\le k\le N-1\}\), must be synthesized. Nevertheless, for chance-constrained problems, this policy class admits convex reparameterizations \cite{skafDesignAffineControllers2010,okamotoOptimalCovarianceControl2018,andersonSystemLevelSynthesis2019}.

\subsection{Deterministic State-History-Affine Formulation}
Under the state-history-affine parameterization in \cref{eq:history_controller}, \cref{eq:problem} becomes a finite-dimensional optimization over $\bm v$ and a block-lower-triangular gain $K$. Define
$\bm{x}:=[x_0^\top,\dots,x_N^\top]^\top$, 
$\bm{u}:=[u_0^\top,\dots,u_{N-1}^\top]^\top$,
$\bm{v}:=[v_0^\top,\dots,v_{N-1}^\top]^\top$.
The control policy \cref{eq:history_controller} is then written in the batch vector form as
\begin{equation*}
	\bm{u} = \bm{v} + K (\bm{x} - \E{\bm{x}}), 
\end{equation*}
where $K$ is block lower triangular with blocks in $\R^{m\times n}$:
\begin{equation} \label{eq:K-definition}
    K = \begin{bmatrix}
    K_{0,0} & 0 & \dots & 0 & 0\\
    K_{1,0} & K_{1,1} & \dots & \vdots &  \vdots\\
    \vdots & \vdots & \ddots & \vdots & \vdots \\
    K_{N-1,0} & K_{N-1,1} & \dots & K_{N-1,N-1} & 0
    \end{bmatrix}.
\end{equation}
With the lifted matrices $A,B,G$ defined in \cref{sec:block-matrix-definitions}, the dynamics are expressed in the batch form \cite{skafDesignAffineControllers2010}
\begin{equation*}
	\bm{x} = A x_0 + B \bm{u} + G \bm{w}.
\end{equation*}
Because $BK$ is strictly block lower triangular, $I-BK$ is invertible. Using the uncontrolled state covariance $S$ defined in \cref{eq:S-definition}, we have \cite{okamotoOptimalCovarianceControl2018}
\begin{subequations} \label{eq:batch-moments}
\begin{align}
	\bm\mu& := \E{\bm x }=A\mu_0+B\bm v,\\
	\bm{v} &= \E{\bm u}, \\
	P_X& :=\cov{\bm x} = (I-BK)^{-1}S(I-BK)^{-\top}, \label{eq:P_X-definition}\\
	P_U& := \cov{\bm u} =KP_XK^\top,\\
	P_{UX}&:=\E{(\bm u-\bm v)(\bm x-\bm\mu)^\top} = KP_X.
\end{align}
\end{subequations}
Define $Q:=\mathrm{blkdiag}(Q_0,\dots,Q_N)$ and $R:=\mathrm{blkdiag}(R_0,\allowbreak \dots, \allowbreak R_{N-1})$. The cost becomes
\begin{equation*}
	J(\bm v,K)=\bm\mu^\top Q\bm\mu+\bm v^\top R\bm v
	+\mathrm{Tr}(QP_X)+\mathrm{Tr}(RP_U).
\end{equation*}
Let $E_k^x$ and $E_k^u$ select $x_k$ and $u_k$ from the corresponding batch vectors. The state and control at each time step are Gaussian due to the linearity of the dynamics and the control policy. Their statistical moments are
\begin{equation} \label{eq:moments}
	\begin{split}
	\mu_k &:= \E{x_k} = E_k^x\bm\mu, \\
	v_k &= \E{u_k} = E_k^u\bm v, \\
	P_{x,k} &:= \cov{x_k} =E_k^xP_X{E_k^x}^\top,\\
	P_{u,k} &:= \cov{u_k} =E_k^uP_U{E_k^u}^\top, \\
	P_{ux,k} &:= \cov{u_k, x_k} =E_k^uP_{UX}{E_k^x}^\top.
	\end{split}
\end{equation}

Thus, \cref{eq:problem} can be reformulated as 
\begin{equation} \label{eq:parameterized-problem}
	\begin{split}
		\min_{\bm{v}, K} \quad & J(\bm{v}, K) \\
		\text{subject to} \quad & g_x(\mu_k, P_{x,k}) \leq 0, \quad k = 1, \dots, N, \\
		& g_u(v_k, P_{u,k}) \leq 0, \quad k = 0, \dots, N-1, \\
		& K \text{ block lower triangular}.
	\end{split}
\end{equation}
This parameterization is exact but generally nonconvex in $K$. Several convex reformulations are available, such as the Youla parameterization \cite{skafDesignAffineControllers2010,okamotoOptimalCovarianceControl2018}, disturbance feedback \cite{goulartOptimizationStateFeedback2006a,balciConstrainedMinimumVariance2024}, and system level synthesis \cite{andersonSystemLevelSynthesis2019,khalilDistributedLocalizedCovariance2025}. These are briefly summarized in \cref{sec:convex-formulation}.
Solving the resulting convex optimization problem provides an optimal solution $\pi^C$ of the state-history-affine formulation.
The analysis in \cref{sec:analysis} is agnostic to the choice of convex reformulation, as long as an optimal solution is obtained.

\subsection{Convex Formulations} \label{sec:convex-formulation}
Here, we summarize three convex reformulations of \cref{eq:batch-moments} that make \cref{eq:parameterized-problem} convex in the decision variables.
The lifted matrices used below ($A$, $B$, $G$, $Z_A$, $Z_B$, $F$, $\Sigma_w$, and $S$) are defined in \cref{sec:block-matrix-definitions}.
\subsubsection{Youla Parameterization \cite{vanhessemConicReformulationModel2002,skafDesignAffineControllers2010,okamotoOptimalCovarianceControl2018}}
This parameterization \cite{youlaModernWienerHopfDesign1976} defines the change of variables
\begin{subequations}\label{eq:K-L}
\begin{align}
	L &= K(I-BK)^{-1}, \label{eq:K2L}\\
	K &= L(I+BL)^{-1}. \label{eq:L2K}
\end{align}
\end{subequations}
This allows the covariance matrices in \cref{eq:batch-moments} to be expressed as
\begin{subequations}\label{eq:batch-moments-Youla}
\begin{align}
	P_X &= (I+BL)S(I+BL)^\top,\\
	P_U &= LSL^\top,\\
	P_{UX} &= LS(I+BL)^\top.
\end{align}
\end{subequations}
The variable $L \in \R^{mN \times n(N+1)}$ is block lower triangular with blocks in $\R^{m \times n}$.

\subsubsection{Disturbance Feedback \cite{goulartOptimizationStateFeedback2006a,bakolasOptimalCovarianceControl2016}}
Define the augmented disturbance vector
\begin{equation}
	\bm d:=
	\begin{bmatrix}
		(x_0-\mu_0)^\top & (G_0w_0)^\top & \cdots & (G_{N-1}w_{N-1})^\top
	\end{bmatrix}^\top.
\end{equation}
Disturbance feedback instead finds a causal gain $K_w$ such that
$\bm u=\bm v+K_w\bm d$ \cite{balciConstrainedMinimumVariance2024}.
Here, $K_w\in\R^{mN\times n(N+1)}$ is block lower triangular with blocks in $\R^{m\times n}$. Since $\cov{\bm d}=\Sigma_w$, the closed-loop covariances are
\begin{subequations}\label{eq:batch-moments-disturbance-feedback}
\begin{align}
	P_X &= (F+BK_w)\Sigma_w(F+BK_w)^\top,\\
	P_U &= K_w\Sigma_wK_w^\top,\\
	P_{UX} &= K_w\Sigma_w(F+BK_w)^\top.
\end{align}
\end{subequations}
The disturbance-feedback and Youla gains satisfy
\begin{equation}\label{eq:Kw-L-relation}
	K_w=K(I-BK)^{-1}F=LF.
\end{equation}
Consequently, the state-history gain is recovered by setting $L=K_wF^{-1}$ in \cref{eq:L2K}.

\subsubsection{System Level Parameterization \cite{andersonSystemLevelSynthesis2019}}
The system level synthesis (SLS) parameterization directly optimizes the causal closed-loop responses
\begin{equation}
	\bm x-\bm\mu=\Phi_x\bm d,
	\qquad
	\bm u-\bm v=\Phi_u\bm d
\end{equation}
subject to the affine subspace constraint \cite{andersonSystemLevelSynthesis2019}
\begin{equation}\label{eq:sls-achievability}
	(I-Z_A)\Phi_x-Z_B\Phi_u=I.
\end{equation}
Here, $\Phi_x\in\R^{n(N+1)\times n(N+1)}$ and $\Phi_u\in\R^{mN\times n(N+1)}$ are block lower triangular, with blocks in $\R^{n\times n}$ and $\R^{m\times n}$, respectively. Their covariance expressions are
\begin{subequations}\label{eq:batch-moments-sls}
\begin{align}
	P_X &= \Phi_x\Sigma_w\Phi_x^\top,\\
	P_U &= \Phi_u\Sigma_w\Phi_u^\top,\\
	P_{UX} &= \Phi_u\Sigma_w\Phi_x^\top.
\end{align}
\end{subequations}
The state-history gain is recovered as $K=\Phi_u\Phi_x^{-1}$. 

The three parameterizations are related by an affine transformation of the decision variables:
\begin{equation}
	\Phi_x=(I+BL)F=F+BK_w,
	\qquad
	\Phi_u=LF=K_w.
\end{equation}

\section{Markovian Policy Construction and Equivalence}
\label{sec:analysis}

The convex reformulations in \cref{sec:description} enable the synthesis of a globally optimal policy $\pi^C$ for \cref{eq:problem} (within the state-history-affine policy class). This section provides a method to construct a Markovian policy $\pi^M$ from $\pi^C$ and shows the equivalence of the two policies.

\subsection{Markovian Policy Construction}

Given a state-history-affine policy, we construct a Markovian policy using \cref{alg:markov-construction}. \cref{thm:markov} shows that the optimal history-feedback policy (which is the input to \cref{alg:markov-construction}) and the Markovian policy (the output) produce identical control actions almost surely.
\begin{algorithm}
	\caption{Markovian Policy Construction}\label{alg:markov}
	\begin{algorithmic}[1]
	\State \textbf{Input:} State-history-affine policy $\pi^C$, consisting of $\{v_k\}_{k=0}^{N-1}$, $\{\mu_k\}_{k=0}^{N-1}$, and $\{K_{k,i}, 0 \leq i \leq k \leq N-1\}$.
	\State Compute the state and control covariances $\{P_{x,k}, P_{u,k}, P_{ux,k}\}_{k=0}^{N-1}$ using \cref{eq:batch-moments,eq:moments}.
	\State Compute the Markovian gains $\{H_k\}_{k=0}^{N-1}$ using
	\begin{equation} \label{eq:H-e-def}
		H_k := P_{ux,k} P_{x,k}^{\dagger}, \quad k = 0, \dots, N-1.
	\end{equation}
	\State \textbf{Output:} Markovian policy $\pi^M$, consisting of $\{v_k\}_{k=0}^{N-1}$, $\{\mu_k\}_{k=0}^{N-1}$, and $\{H_k\}_{k=0}^{N-1}$.
	\end{algorithmic}
	\label{alg:markov-construction}
\end{algorithm}

Consider $\pi^C$ and $\pi^M$, with the latter constructed from the former using \cref{alg:markov-construction}.
Define the stochastic error term $e_k$:
\begin{equation} \label{eq:e-def}
	e_k = \pi_k^C(x_0, \dots, x_k) - \pi_k^M(x_k), \quad k = 0, \dots, N-1.
\end{equation}
From \cref{eq:markov-controller,eq:history_controller}, 
\begin{align}
	e_k &= [v_k + \sum_{i=0}^{k} K_{k,i} (x_i - \mu_i)] - [v_k + H_k (x_k - \mu_k)] \nonumber \\
	&= \delta u_k - H_k \delta x_k \label{eq:e_k-with-delta}
\end{align}
where 
\begin{equation}
	\delta x_k = x_k - \mu_k, \quad \delta u_k = u_k - v_k
\end{equation}
are the deviations of the state and control from their respective means under the policy $\pi^C$. Note that $\E{e_k} = 0$ because $\E{\delta u_k} = 0$ and $\E{\delta x_k} = 0$.

The error term $e_k$ satisfies the following properties, which are used in the proof of \cref{thm:markov}.
\begin{lemma} \label{lem:ek-properties}
	Define $H_k$ and $e_k$ as in \cref{eq:H-e-def,eq:e-def}. Then, 
	\begin{enumerate}[label=(\alph*)]
		\item $\E{e_k \delta x_k^\top} = 0$,
		\item $P_{e,k} = P_{u,k} - H_k P_{ux,k}^\top = P_{u,k} - P_{ux,k} P_{x,k}^{\dagger} P_{ux,k}^\top.$
	\end{enumerate}
\end{lemma}
\begin{proof}

	[Proof of (a)]
	\begin{align*}
		\E{e_k \delta x_k^\top}
		&= \E{\delta u_k \delta x_k^\top} - H_k \E{\delta x_k \delta x_k^\top} \\
		&= P_{ux,k} - P_{ux,k} P_{x,k}^{\dagger} P_{x,k} \\
		&= 0
	\end{align*}
	Most of the proof for the last equality follows in the proof of \cite[Proposition 6]{liuOptimalCovarianceSteering2025}, but we include it here for completeness. Let us drop the $k$ subscript for brevity. 
	Let $z \in \ker(P_x)$. Then, $\delta x^\top z = 0$ a.s., which implies $P_{ux} z = \E{\delta u \delta x^\top} z = \E{\delta u (\delta x^\top z)} = 0$. Thus, $\ker(P_x) \subseteq \ker(P_{ux})$. This is equivalent to $P_{ux} = P_{ux} P_x^{\dagger} P_x$.

	[Proof of (b)] Take the covariance of both sides of \cref{eq:e_k-with-delta}, substituting the definition of $H_k$.
	Remove the $k$ subscript for brevity. 
		\begin{align}
		P_{e}
		&= \E{\delta u \delta u^\top}
		- \E{\delta u \delta x^\top} P_x^{\dagger} P_{ux}^\top \nonumber \\
		&\quad - P_{ux} P_x^{\dagger} \E{\delta x \delta u^\top} 
		+ P_{ux} P_x^{\dagger} \E{\delta x \delta x^\top}
		P_x^{\dagger} P_{ux}^\top \nonumber \\
		&= P_u - P_{ux} P_x^{\dagger} P_{ux}^\top \label{eq:P_e-equality}
		\end{align}
	where the property $P_x^{\dagger} P_x P_x^{\dagger} = P_x^{\dagger}$ is used in the last equality.
\end{proof}


To isolate the effect of $e_k$ at time $k$, define the policy $\pi'(k)$ that uses $\pi^C$ up to time step $k-1$ and the corresponding $\pi^M$ from time step $k$ onward:
\begin{equation} \label{eq:u_t-prime}
			u_t'(x_{0:t}; k) = \begin{cases}
				\pi_t^C(x_{0:t}), & t<k, \\
				\pi_t^M(x_t), & t\geq k.
			\end{cases}
\end{equation}
Next, in \cref{lem:covariance-diff}, we use this policy to assert that when later controls are computed using the Markovian policy, the effect of $e_k$ on the state and control covariances at later time steps can be expressed in terms of its covariance $P_{e,k}$.

\begin{lemma} \label{lem:covariance-diff}
	Consider the policy $\pi^C$ of the form \cref{eq:history_controller} and $\pi'(k)$ of the form \cref{eq:u_t-prime}.
	Let $\delta x_t' = x_t' - \mu_t$ and $\delta u_t' = u_t' - v_t$ denote the state and control deviations under $\pi'$. Define $e_k$ as in \cref{eq:e-def} and its covariance as $P_{e,k} := \E{e_k e_k^\top}$. Finally, assume that
	\begin{equation}
		\delta u_t = H_t \delta x_t, \quad \text{a.s.}\quad t=k+1, \dots, N-1,
	\end{equation}
	where $\delta x_t$ and $\delta u_t$ are defined above.

	Then, 
	\begin{align}
		\E{x_t'} &= \E{x_t}, &&t=k,\dots,N\\
		\E{u_t'} &= \E{u_t}, &&t = k,\dots,N-1 \\
		P_{x,t} &= P_{x,t}' + T_{t,k}P_{e,k}T_{t,k}^\top, &&t=k+1,\dots,N \label{eq:P_xt-diff}\\
		P_{u,k} &= P_{u,k}' + P_{e,k}, \label{eq:P_uk-diff}\\
		P_{u,t} &= P_{u,t}' + H_t T_{t,k} P_{e,k} T_{t,k}^\top H_t^\top, &&t=k+1,\dots,N-1\label{eq:P_ut-diff}
	\end{align}
	where 
	\begin{equation*}
		T_{k+1,k} = B_k, \qquad T_{t+1,k} = (A_t + B_t H_t) T_{t,k}.
	\end{equation*}
\end{lemma}
\begin{proof}
	The equivalence of the means is straightforward, since the policies share $v$ and $\mu$.
	Because $\pi_t^C(x_{0:t}) = \pi_t'(x_{0:t})$ for $t < k$, we have $\delta x_k' = \delta x_k$.
	Since $u_k' = u_k^M$, 
	\begin{equation}
		\delta u_k - \delta u_k' = [\sum_{i=0}^{k} K_{k,i} (x_i - \mu_i)] - [H_k (x_k - \mu_k)] = e_k.
	\end{equation}
	Rearranging,
	\begin{equation} \label{eq:u_k-diff}
		\delta u_k = \delta u_k' + e_k.
	\end{equation}
	Then, 
	\begin{align*}
		\delta x_{k+1} &= A_k \delta x_k + B_k \delta u_k + G_k w_k \\ 
		& = (A_k + B_k H_k) \delta x_k + B_k e_k + G_k w_k \\
		\delta x_{k+1}' &= A_k \delta x_k' + B_k \delta u_k' + G_k w_k \\
		& = (A_k + B_k H_k) \delta x_k' + G_k w_k \\
		& \implies \delta x_{k+1} = \delta x_{k+1}' + B_k e_k
	\end{align*}
	Repeating this argument for $t=k+1,\dots,N-1$ gives
	\begin{equation} \label{eq:x_t-diff}
		\delta x_t = \delta x_t' + T_{t,k} e_k, \quad t = k+1, \dots, N.
	\end{equation}
	For $t=k+1,\dots,N-1$, 
	\begin{align} \label{eq:u_t-diff}
		\delta u_t = H_t \delta x_t 
		= H_t (\delta x_t' + T_{t,k} e_k)
		= \delta u_t' + H_t T_{t,k} e_k. 
	\end{align}
	Next, use that $\delta x_t'$ is an affine function of $\delta x_k$ and $\{w_i\}_{i=k}^{t-1}$, and 
	\begin{equation}
		\cov{e_k, \delta x_k} = 0, \quad \cov{e_k, w_i} = 0, \ i=k,\dots,N-1,
	\end{equation}
	where the first equality is from \cref{lem:ek-properties} and the second is from the independence of $e_k$ and $w_i$.
	This gives
	\begin{equation} \label{eq:x_t-e_k-indep}
		\cov{\delta x_t', e_k} = 0, \quad t=k,\dots,N.
	\end{equation}
	Combining \cref{eq:x_t-diff} and \cref{eq:x_t-e_k-indep} gives 
	\cref{eq:P_xt-diff}.
	Since 
	\begin{align}
		\E{\delta u_t' (e_k^\top)} = \E{H_t \delta x_t' (e_k^\top)} = H_t \E{\delta x_t' e_k^\top} = 0,
	\end{align}
	\cref{eq:u_k-diff,eq:u_t-diff} and \cref{eq:x_t-e_k-indep} give
	\cref{eq:P_uk-diff,eq:P_ut-diff}.
\end{proof}

\cref{lem:covariance-diff} provides the equations for the covariances of the state and control under $\pi^C$ and $\pi'(k)$, which are used in \cref{lem:residual-dominance} to show that $\pi'(k)$ is better than or equal to $\pi^C$ in terms of cost and feasibility.
\begin{lemma} \label{lem:residual-dominance}
	Define the control policies $\pi^C$ and $\pi'(k)$ as in \cref{lem:covariance-diff}. Let $J$ and $J'(k)$ be the costs of $\pi^C$ and $\pi'(k)$, respectively. Then, 
	\begin{enumerate}[label=(\alph*)]
		\item $J - J'(k) \geq 0.$
		\item If $\pi^C$ is feasible for \cref{eq:problem}, then $\pi'(k)$ is also feasible.
	\end{enumerate}
\end{lemma}
\begin{proof}

	[Proof of (a)] Let $P_{x,t}$, $P_{u,t}$ be the covariances of the state and control under $\pi^C$. The cost for $\pi^C$ is written as
	\begin{equation} \label{eq:cost-causal-stepwise}
		\begin{split}
		&J = \sum_{t=0}^N \mu_t^\top Q_t \mu_t 
		+ \sum_{t=0}^{N-1} v_t^\top R_t v_t \\
		& \quad + \sum_{t=0}^{N} \mathrm{Tr}(Q_t P_{x,t}) + \sum_{t=0}^{N-1} 
		\mathrm{Tr}(R_t P_{u,t}).
		\end{split}
	\end{equation}
	$J'(k)$ can be written with $P_{x,t}'$ and $P_{u,t}'$ in place of $P_{x,t}$ and $P_{u,t}$. Taking the difference and using \cref{eq:P_xt-diff,eq:P_uk-diff,eq:P_ut-diff} gives
	\begin{align*}
		J-J'(k)={}&\mathrm{Tr}(R_kP_{e,k})
		+\sum_{t=k+1}^{N}\mathrm{Tr}\!\left(Q_tT_{t,k}P_{e,k}T_{t,k}^\top\right) \\
		&+\sum_{t=k+1}^{N-1}\mathrm{Tr}\!\left(R_tH_tT_{t,k}P_{e,k}T_{t,k}^\top{H_t}^\top\right).
	\end{align*}
	Since $R_t \succ 0$ and $Q_t \succeq 0$, the result follows.

	[Proof of (b)] From \cref{eq:P_xt-diff,eq:P_uk-diff,eq:P_ut-diff}, 
	\begin{equation*}
		P_{x,t} \succeq P_{x,t}', \quad P_{u,t} \succeq P_{u,t}'.
	\end{equation*}
	Since the constraints in \cref{eq:problem} are covariance-monotonic, if $\pi^C$ is feasible, then $\pi'(k)$ is also feasible.
\end{proof}
Note that both \cref{lem:covariance-diff} and \cref{lem:residual-dominance} hold for any $\pi^C$, regardless of whether $\pi^C$ is optimal or not.

\subsection{Almost-Sure Equivalence of the Constructed Policy}

Now, suppose that $\pi^C$ is an optimal policy for \cref{eq:problem}. Then we can apply backward induction and recursively use \cref{lem:residual-dominance} to show that the error term $e_k$ (which, recall, is the difference between the control actions at time $k$ under $\pi^C$ and $\pi^M$) must be zero almost surely. By iterating through $\pi'(k)$ backward from $k=N-1$, we finally reach $\pi'(0)$, which is equivalent to $\pi^M$. This is shown in \cref{thm:markov}.
\begin{theorem}  \label{thm:markov}
	Let $\pi^C$ be an optimal solution of the unrestricted state-history-affine formulation \cref{eq:parameterized-problem}, and let $\pi^M$ be the Markovian policy constructed from $\pi^C$ using \cref{alg:markov-construction}\footnote{If, at any time step $k$, the state covariance $P_{x,k}$ is rank-deficient, then the Markovian gain $H_k$ obtained from \cref{eq:H-e-def} in \cref{alg:markov-construction} is not unique and can be chosen arbitrarily in the null space of $P_{x,k}$ without affecting the claim of the theorem. See \cref{sec:general-solution-singular-covariance} for details.}. Then,
	\begin{equation}
		\pi_k^C(x_0, \dots, x_k) = \pi_k^M(x_k) \quad a.s., \qquad k=0,\dots,N-1.
	\end{equation}
\end{theorem}
\begin{proof}
	From $\cref{eq:e-def}$,
	we want to show that $e_k=0$ a.s. for all $k$. 
	Since $\E{e_k} = 0$, it suffices to show that $P_{e,k}=0$ for all $k$. We proceed by backward induction.

	\textbf{Induction hypothesis:} At time step $k < N-1$, assume that $e_t = 0$ a.s. for $t = k+1, \dots, N-1$. 
	
	The induction hypothesis satisfies the assumption of \cref{lem:residual-dominance}.
	Let us construct policy $\pi'(k)$ as in \cref{eq:u_t-prime}, with covariance $P_{x,t}'$ and $P_{u,t}'$.
	If $P_{e,k} \neq 0$, then from \cref{lem:residual-dominance}, $J - J' > 0$ while $\pi'(k)$ is feasible. This contradicts the optimality of $\pi^C$. Therefore, $P_{e,k} = 0$.
	
	The base case $k=N-1$ automatically satisfies the induction hypothesis, since there are no later time steps. 

\end{proof}

\subsection{A Stronger Result for Nondegenerate Disturbances}
\cref{thm:markov} establishes the equivalence of the control actions implemented by the two policies. 
\cref{lem:equivalence-stochastic-support} (specifically, (c)) provides a formula for numerically verifying the equivalence of the two policies. We also use this lemma to show a stronger result for \cref{thm:markov} in a special case, in \cref{thm:nondegenerate}.

\begin{lemma} \label{lem:equivalence-stochastic-support}
	Let $K$ and $P_X$ be defined as in \cref{eq:K-definition,eq:P_X-definition}, respectively, and define
	\begin{equation} \label{eq:K_M}
		K_M := \begin{bmatrix}
			\mathrm{blkdiag}(H_0, H_1, \dots, H_{N-1}) & 0_{mN \times n}
		\end{bmatrix},
	\end{equation}
	and $\delta \bm{x} = \bm{x} - \E{\bm{x}}$.
	The following statements are equivalent:
	\begin{enumerate}[label=(\alph*)]
		\item $\delta u_k = \sum_{i=0}^{k} K_{k,i} \delta x_i = H_k \delta x_k$ a.s., for all $k = 0, \dots, N-1$.
		\item $(K - K_M) \delta \bm{x} = 0$ a.s.
		\item $ (K - K_M) P_X^{1/2} = 0 $
	\end{enumerate}
\end{lemma}
\begin{proof}

	[(a) $\Leftrightarrow$ (b)] By stacking the equations in (a) vertically for all $k$, we obtain (b). Conversely, the $k$-th block row of (b) is equivalent to the $k$-th equation in (a).

	[(b) $\Leftrightarrow$ (c)] $(\Rightarrow)$: Let $D = K - K_M$. (b) implies
	\begin{equation*}
		\cov{D \delta \bm{x}}
		= D P_X D^\top
		= (D P_X^{1/2}) (D P_X^{1/2})^\top
		= 0.
	\end{equation*}
	This implies $D P_X^{1/2} = 0$.

	\noindent ($\Leftarrow$): Suppose $D P_X^{1/2} = 0$. Then, $\cov{D \delta \bm{x}} = 0$. Since $\E{D \delta \bm{x}} = 0$,
	\begin{equation*}
		\E{\|D \delta \bm{x}\|_2^2}
		= \mathrm{Tr}(\cov{D \delta \bm{x}})
		= 0.
	\end{equation*}
	Since $\|D \delta \bm{x}\|_2^2 \geq 0$, $D \delta \bm{x} = 0$ a.s.
\end{proof}

\cref{lem:PD-equivalence} connects the positive definiteness of the disturbance (initial state and process noise) covariances to the positive definiteness of the state history covariance. We will also use this for the proof of \cref{thm:nondegenerate}. 
\begin{lemma} \label{lem:PD-equivalence}
	For \cref{eq:problem}, let $S$ and $P_X$ be defined as in \cref{eq:S-definition} and \cref{eq:P_X-definition}.
	The following statements are equivalent:
	\begin{enumerate}[label=(\alph*)]
		\item $P_0 \succ 0$ and $G_k G_k^\top \succ 0$ for all $k = 0, \dots, N-1$.
		\item $S \succ 0$.
		\item $P_X \succ 0$.
	\end{enumerate}
\end{lemma}
\begin{proof}
		Let $Z_A$, $\Sigma_w$, and $F$ be defined as in \cref{sec:block-matrix-definitions}.

	[(a) $\Leftrightarrow$ (b)]: Since $Z_A$ is strictly block lower triangular,
	\begin{equation*}
		F=(I-Z_A)^{-1}=I+Z_A+\dots+Z_A^N
	\end{equation*}
	is invertible \cite[Chapter 3.1.7]{golubMatrixComputations2013}. Moreover, $S=F\Sigma_wF^\top$ by \cref{eq:S-definition}.
	Therefore, (a) is equivalent to $\Sigma_w \succ 0$, which is equivalent to $S \succ 0$.
	
	[(b) $\Leftrightarrow$ (c)]: This follows from
	$P_X = (I - B K)^{-1} S (I - B K)^{-\top}$ and invertibility of $(I - B K)$.
\end{proof}

\cref{lem:equivalence-stochastic-support,lem:PD-equivalence} combine to show the equivalence of the gains for $\pi^C$ and $\pi^M$ in a special case where the covariance matrices of the initial state and the process noise are positive definite. 
\begin{theorem} \label{thm:nondegenerate}
	Let $P_0 \succ 0$ and $G_k G_k^\top \succ 0$ for all $k = 0, \dots, N-1$. Then, let $K$ be an optimal solution to \cref{eq:parameterized-problem}. Let $K_M$ be obtained via \cref{alg:markov-construction} and \cref{eq:K_M} from $K$. Then,
	\begin{equation}
		K = K_M.
	\end{equation}
\end{theorem}
\begin{proof}
	Equivalence of (a) and (c) in \cref{lem:PD-equivalence} implies $P_X \succ 0$. 
	From \cref{thm:markov} and equivalence of (a) and (c) in \cref{lem:equivalence-stochastic-support}, we have $(K - K_M) P_X^{1/2} = 0$. Combining, we have $K = K_M$.
\end{proof}
The result here is stronger than \cref{thm:markov} in that the functional form of the optimal policy is deterministically equivalent, as opposed to the control actions being (almost surely) equivalent in \cref{thm:markov}. 
When the assumptions of \cref{thm:nondegenerate} are satisfied, the optimal policy obtained by solving \cref{eq:parameterized-problem} is readily Markovian; $K$ is block diagonal, and we can use $H_k = K_{k,k}$ for all $k$ instead of \cref{eq:H-e-def}. \cref{eq:H-e-def} can also be used as it gives the same value as $K_{k,k}$.

\subsection{Effect of Limiting the Solution Space} \label{sec:limiting-sol-space}
Convex reformulations of \cref{eq:parameterized-problem} involve a change of variables from $K$ to another variable. Several works recommend restricting the solution space in this variable for computational efficiency \cite{skafDesignAffineControllers2010,okamotoOptimalStochasticVehicle2019,balciConstrainedMinimumVariance2024}.
However, such restrictions may exclude the recovered Markovian controller from the feasible set, meaning that $\pi^M$ obtained via \cref{alg:markov-construction} may not be equivalent to $\pi^C$ in the sense of \cref{thm:markov}.
The numerical section illustrates this.

\textbf{Example:} The Youla parameterization in \cref{eq:K-L} is used in \cite{skafDesignAffineControllers2010,bakolasOptimalCovarianceControl2016,okamotoOptimalCovarianceControl2018}. Since $K$ is block lower triangular, $L$ is also block lower triangular; one may restrict $L$ to be, for example, block diagonal \cite{skafDesignAffineControllers2010}. However, if we 1) obtain $L$ in this restricted solution space, 2) convert it to $K$ via \cref{eq:L2K}, 3) obtain $H_k$ and $K_M$ via \cref{alg:markov-construction} and \cref{eq:K_M}, and 4) map $K_M$ back to $L$ via \cref{eq:K2L}, the resulting $L$ is not block diagonal in general. Therefore, the recovered Markovian controller need not belong to the restricted feasible set expressed in the $L$ variable.

\subsection{Discussion}

The results of this section invite new insights into the covariance steering literature and a broader class of finite-horizon stochastic optimal control problems. 

\subsubsection{Expanded Class of Problems with Markovian Optimal Policies}
\cite{liuOptimalCovarianceSteering2025} proves the optimality of policies of the form \cref{eq:markov-controller} for unconstrained covariance steering.
For constrained problems, \cref{thm:markov} makes a different statement: every optimal solution of the unrestricted state-history-affine formulation \cref{eq:parameterized-problem} admits an equivalent Markovian realization. It does not establish that affine policies are globally optimal for constrained problems; this remains unknown.


\subsubsection{Online Storage Complexity}
For a horizon of length $N$, the state-history-affine policy contains
$N(N+1)/2$ gain matrices $K_{k,i}\in\R^{m\times n}$, whereas the
recovered Markovian policy contains only $N$ gains
$H_k\in\R^{m\times n}$. The required gain storage therefore decreases
from $N(N+1)mn/2$ to $Nmn$, and online evaluation requires only the
current state deviation.
The offline synthesis problem remains unaffected, since \cref{alg:markov-construction} is a post-processing step.

\subsubsection{Connection to Previous Work in Covariance Steering} 


The batch formulations that we treat in this work
\cite{okamotoOptimalCovarianceControl2018,
balciConstrainedMinimumVariance2024} use state-history-affine policy
parameterizations for convexity. In contrast, covariance-propagation
formulations
\cite{liuOptimalCovarianceSteering2025,
rapakouliasDiscreteTimeOptimalCovariance2023} use Markovian policies
directly. Our previous comparison
\cite{kumagaiSquareRootFactorizedCovariance2026} identifies the growing storage requirements as an implementation disadvantage of the batch
formulation. \cref{thm:markov} shows that for problems of the form \cref{eq:problem}, after solving the unrestricted
state-history-affine problem, this history dependence is not required for
online implementation.

Since \cref{thm:markov} shows that the batch formulation admits the same Markovian implementation form as the covariance-propagation formulations for \cref{eq:problem}, their differences can be understood through how much each formulation lifts the solution space for convexity. The Markovian gain parameterization has dimension $Nmn$. The batch formulation lifts the solution space to state-history-affine policies, which has dimension $N(N+1)mn/2$, while the covariance-propagation formulation lifts the solution space to the state-control augmented covariance matrix, which has dimension $N(n+m)(n+m+1)/2$. For sufficiently large $N$, the batch formulation has a larger solution space than the covariance-propagation formulation. An interpretation of the convexity of chance constraints in the batch formulation may be that it is a consequence of the larger solution space.

\subsubsection{Extension to General Finite-Horizon Stochastic Control}
Although the focus of this work is specifically on covariance steering, the use of state-history-affine policies is common in the finite-horizon control of uncertain systems, such as robust control \cite{goulartOptimizationStateFeedback2006a} and system level synthesis \cite{andersonSystemLevelSynthesis2019}. The results of this work may be applicable to these problems as well, but further investigation is needed.

\section{Extension to Output Feedback}
We extend the results to controllers that have access only to noisy state measurements. This extension requires slight modifications to account for the state estimation process.

Consider the observation process
\begin{equation*}
y_k = C_k x_k + D_k \eta_k,
\end{equation*}
where $\eta_k\sim\normal(0,I)$ is i.i.d. and independent of $x_0$ and $\{w_k\}_{k=0}^{N-1}$.
Define the filtration $\{\mathcal F_k\}_{k=-1}^N$ by $\mathcal{F}_{-1} = \sigma(\hat{x}_0^-)$ and
$\mathcal F_k:=\sigma(\hat{x}_0^-,y_0,\dots,y_k)$, where $\hat{x}_0^-$ is the initial state estimate and $\hat x_k:=\E{x_k\mid\mathcal F_k}$ \cite{ridderhofChanceConstrainedCovariance2020}. 
The initial state estimate $\hat{x}_0^-$ is assumed to be Gaussian and independent of the initial estimation error $\widetilde x_0^-:=x_0-\hat x_0^-$. 
The estimate-history-affine and Markovian (i.e., current-estimate-affine) controllers are defined as
\begin{equation} \label{eq:causal-controller-output}
	u_k = \pi_k^C(\hat{x}_0,\hat{x}_1,\dots,\hat{x}_k) = v_k+\sum_{i=0}^k K_{k,i}\bigl(\hat x_i-\E{\hat x_i}\bigr),
\end{equation}
and
\begin{equation} \label{eq:markov-controller-output}
	u_k = \pi_k^M(\hat x_k) = v_k+H_k\bigl(\hat x_k-\E{\hat x_k}\bigr),
\end{equation}
respectively. 

Then, the moment equations for the batch state and control \cref{eq:batch-moments} and those for the individual time steps \cref{eq:moments} are modified to account for the estimation process in \cref{lem:output-fb-moments}.
\begin{lemma}
\label{lem:output-fb-moments}
	Let $\widetilde x_k:=x_k-\hat x_k$, $\widetilde P_k:=\cov{\widetilde x_k}$,
	$P_{\hat x,k}:=\cov{\hat x_k}$, and
	$P_{u\hat x,k}:=\cov{u_k,\hat x_k}$. Writing
	$\hat{\bm x}:=[\hat x_0^\top,\dots,\hat x_N^\top]^\top$, its covariance is
	obtained from \cref{eq:batch-moments} as
	\begin{equation}
		\begin{aligned}
			P_{\widehat X}&=(I-BK)^{-1}\widehat S(I-BK)^{-\top},\\
			P_{\hat x,k}&=E_k^xP_{\widehat X}{E_k^x}^\top,\\
			P_{u\hat x,k}&=E_k^u K P_{\widehat X}{E_k^x}^\top,
		\end{aligned}
		\label{eq:output-fb-estimate-batch-covariance}
	\end{equation}
	where $\widehat S$ is a constant matrix formed from the initial estimate covariance and the Kalman innovation covariances. Then the moments in \cref{eq:moments} under the controller \cref{eq:causal-controller-output} satisfy
	\begin{equation}
		\begin{split}
		\mu_k &= \E{x_k}=\E{\hat x_k}, \qquad
		P_{x,k} = P_{\hat x,k}+\widetilde P_k,\\
		v_k &= \E{u_k}, \qquad
		P_{ux,k}=P_{u\hat x,k}. \label{eq:output-fb-control-moments}
		\end{split}
	\end{equation}
\end{lemma}
\begin{proof}
	The result follows from the independence of the estimation error $\widetilde x_k$ and the Kalman estimate $\hat x_k$ \cite{ridderhofChanceConstrainedCovariance2020}.
\end{proof}

Using \cref{lem:output-fb-moments}, we can now modify \cref{eq:H-e-def} of \cref{alg:markov} and the proof of \cref{thm:markov} for the output-feedback case.
\begin{theorem} \label{prop:output-fb}
	Let $\pi^C$ be an optimal solution of the unrestricted estimate-history-affine formulation. Applying
	\cref{alg:markov} to the estimate process, using the moments in
	\cref{lem:output-fb-moments} and replacing \cref{eq:H-e-def} with
	\begin{equation} \label{eq:H_k-output}
		H_k=P_{u\hat x,k}{P_{\hat x,k}}^{\dagger}, \qquad k=0,\dots,N-1,
	\end{equation}
	yields a current-estimate-affine policy satisfying
	\begin{equation}
		\pi_k^C(\hat x_0,\dots,\hat x_k)
		=\pi_k^M(\hat x_k)\quad\text{a.s.}, \qquad k=0,\dots,N-1.
	\end{equation}
\end{theorem}

\begin{proof}
	The Kalman estimate is a linear controlled Markov process with
	controller-independent Gaussian innovations, and $\widetilde P_k$ is
	deterministic and controller-independent
	\cite{ridderhofChanceConstrainedCovariance2020}. By
	\cref{lem:output-fb-moments}, the state/control means and the control covariances retain the same form as in \cref{eq:moments}. The only difference is the constant contribution $\widetilde P_k$ to the state covariance. Hence, the covariance contribution to the state cost
	in \cref{eq:cost-causal-stepwise} separates as
	\begin{equation*}
		\sum_{k=0}^N\mathrm{Tr}(Q_kP_{x,k})
		=\sum_{k=0}^N\mathrm{Tr}(Q_kP_{\hat x,k})
		+\sum_{k=0}^N\mathrm{Tr}(Q_k\widetilde P_k).
	\end{equation*}
	The first term is controller-dependent, whereas the second is constant.
	Each state constraint in \cref{eq:problem} becomes
	$g_x(\E{\hat x_k},P_{\hat x,k}+\widetilde P_k)\leq0$, which remains
	covariance-monotone in $P_{\hat x,k}$. Hence \cref{thm:markov} applies to
	the estimate process and gives the result.
\end{proof}

\section{Extension to a Value-at-Risk Surrogate} \label{sec:var}

We have focused on the standard quadratic cost thus far. 
Here, we show that when replacing the cost with a value-at-risk (VaR) surrogate, the Markovian policy from \cref{alg:markov-construction} is optimal, and interestingly, can have ``smaller'' covariances than for the original history-feedback controller.

For a scalar random variable $Z$, its $\beta$-VaR for $\beta \in (0,1)$ is defined as \cite{rockafellarOptimizationConditionalValueatRisk2000}
\begin{equation*}
	\mathrm{VaR}_\beta(Z) = \inf \{ z \in \R : \P{Z \leq z} \geq \beta \}.
\end{equation*}
As a quantile-based risk measure, VaR and its variants are often used in domains such as finance \cite{basakValueatRiskBasedRiskManagement2001}, robotics \cite{renganathanRiskBoundedNonlinear2023}, electric power systems \cite{kleindorferMultiPeriodVaRConstrainedPortfolio2005}, and aerospace \cite{cangahualaEuropaClipperMission2025}.
Especially for spacecraft control applications, 
minimizing the VaR of the control Euclidean norm is common \cite{oguriChanceConstrainedControlSafe2024a,kumagaiRobustCislunarLowThrust2025,kongCertifiedStochasticControl2026}. For a Gaussian control input $u_k \sim \normal(v_k, P_{u,k})$, the $(1-\gamma)$-VaR has an upper bound \cite{oguriChanceConstrainedControlSafe2024a}
\begin{align*}
	\mathrm{VaR}_{1-\gamma}(\|u_k\|_2)
	&\leq \|v_k\|_2
	+ \alpha_{m, \gamma} \sqrt{\lambda_{\max}(P_{u,k})}
\end{align*}
derived using the same machinery as in \cref{eq:cc-norm}.
The right-hand side can be converted to convex expressions with the same change of variables discussed in \cref{sec:description} \cite{oguriChanceConstrainedControlSafe2024a}. For example, under the change of variables in \cref{eq:K-L}, 
\begin{equation*}
	P_{u,k}=E_k^uLSL^\top {E_k^u}^\top, \qquad
	\sqrt{\lambda_{\max}(P_{u,k})} = \| E_k^u L S^{1/2} \|_2,
\end{equation*}
which is convex in $L$.
Motivated by these stagewise upper bounds, we consider the following additive surrogate \cite{oguriChanceConstrainedControlSafe2024a}:
\begin{equation} \label{eq:control-var}
	\widehat J_{\mathrm{VaR}}(\pi)
	:=\sum_{k=0}^{N-1}
	\left(
		\|v_k\|_2
		+\alpha_{m, \gamma}\sqrt{\lambda_{\max}(P_{u,k})}
	\right).
\end{equation}

\begin{proposition} \label{prop:var-markov-recovery}
	Consider \cref{eq:problem} with objective \cref{eq:control-var}. If
	$\pi^C$ is an optimal solution of the unrestricted state-history-affine formulation, then the Markovian
	controller $\pi^M$ constructed by \cref{alg:markov-construction} attains the same optimal objective value. Moreover, it preserves the state and control means and satisfies
	\begin{equation}
		P_{x,k}^M\preceq P_{x,k}^C,
		\qquad
		P_{u,k}^M\preceq P_{u,k}^C
	\end{equation}
	for all $k$, where the superscripts $C$ and $M$ denote the covariances under $\pi^C$ and $\pi^M$, respectively.
\end{proposition}

\begin{proof}
    The controller construction, mean preservation, covariance dominance, and
    feasibility follow exactly as in \cref{thm:markov}, with feasibility given
    by \cref{lem:residual-dominance}(b). The only change is the cost comparison
    in \cref{lem:residual-dominance}(a). For the replacement $\pi'$ at time
    $k$, the control means are unchanged and
    \begin{equation*}
		\begin{split}
        &\widehat J_{\mathrm{VaR}}(\pi^C)
        -\widehat J_{\mathrm{VaR}}(\pi') \\
        &\quad =\alpha_{m,\gamma}\sum_{t=k}^{N-1}
        \left(
            \sqrt{\lambda_{\max}(P_{u,t})}
            -\sqrt{\lambda_{\max}(P_{u,t}')}
        \right) \\
        &\quad \geq 0.
		\end{split}
    \end{equation*}
    The last inequality follows from the fact that $P_{u,t}'\preceq P_{u,t}$ by
    \cref{lem:covariance-diff}, and the maximum eigenvalue is
    nondecreasing in the Loewner order. Applying the construction of \cref{thm:markov}
    therefore gives a feasible $\pi^M$ satisfying the stated covariance
    inequalities and
    \begin{equation} \label{eq:var-cost-ineq}
        \widehat J_{\mathrm{VaR}}(\pi^M)
        \leq\widehat J_{\mathrm{VaR}}(\pi^C).
    \end{equation}
    Optimality of $\pi^C$ implies equality, so $\pi^M$ is also optimal.
\end{proof}

\begin{remark}
	Unlike \cref{thm:markov}, \cref{prop:var-markov-recovery} does not imply that $\pi^C$ and $\pi^M$ produce the same covariance trajectories.
	Define 
	\begin{equation*}
	\Delta P_{u,k} := P_{u,k}^C - P_{u,k}^M \succeq 0,
	\end{equation*}
	which corresponds to the final terms in \cref{eq:P_uk-diff,eq:P_ut-diff}.
	The maximum eigenvalue operator appearing in \cref{eq:control-var} is not strictly increasing under every nonzero positive-semidefinite perturbation. For example,
	\begin{equation*}
		P_{u,k}^M=\mathrm{diag}(1,0),
		\qquad
		\Delta P_{u,k}=\mathrm{diag}(0,1/2)
	\end{equation*}
	give 
	\begin{equation*}
		\lambda_{\max}(P_{u,k}^C) = \lambda_{\max}(P_{u,k}^M + \Delta P_{u,k})
		= 1 = \lambda_{\max}(P_{u,k}^M),
	\end{equation*}
		i.e., $\Delta P_{u,k}$ need not be zero for the maximum eigenvalues to agree. Hence, the state-history-affine and recovered Markovian controllers may have different state and control covariance histories while attaining the same objective value and satisfying the same set of constraints. However, due to equality in \cref{eq:var-cost-ineq}, the maximum eigenvalues of the control covariances must be equal for all $k$.

	Adding a control-quadratic term $\E{u_k^\top R_k u_k}$ for each $k$ with $R_k \succ 0$ to \cref{eq:control-var} \cite{kumagaiRobustCislunarLowThrust2025} penalizes $P_{u,k}$ ``from all directions'' and forces $P_{e,k}=0$ for all $k$, allowing the statement of \cref{thm:markov} to apply.

\end{remark}

\begin{remark}[Extension to State VaR Surrogate and Output Feedback]
	The conclusion of \cref{prop:var-markov-recovery} holds as stated when
	the cost function additionally penalizes a state VaR surrogate. 
	Minimizing the VaR (or conditional VaR) of the state is often considered in risk-sensitive planning and control \cite{ningOnlineLearningBased2021,chapmanRiskSensitiveSafetyAnalysis2022,millerConvexComputationValueatRisk2025,qiSafeManeuverPlanning2026}.
	For
    $x_k\sim\normal(\mu_k,P_{x,k})$, consider the surrogate for the VaR cost of the state Euclidean norm
    \begin{equation*}
        \widehat J_{\mathrm{VaR}}^x(\pi)
        :=
        \sum_{k=0}^{N}
        \left(
            \|\mu_k\|_2
            +\alpha_{n,\gamma_x}
            \sqrt{\lambda_{\max}(P_{x,k})}
	        \right),
	    \end{equation*}
		for some $\gamma_x\in(0,1)$. This surrogate can also be expressed in convex form via \cref{eq:K-L}.
		
		The output-feedback analogue follows by applying the same argument to the
		estimate process, as in \cref{prop:output-fb}. By
		\cref{lem:output-fb-moments}, the controller-dependent moments retain the
		same form and the estimation-error covariance is controller-independent;
		hence the proof of \cref{prop:var-markov-recovery} is unchanged.
\end{remark}

\begin{figure*}[!b]
	\centering
	\includegraphics[width=0.94\textwidth]{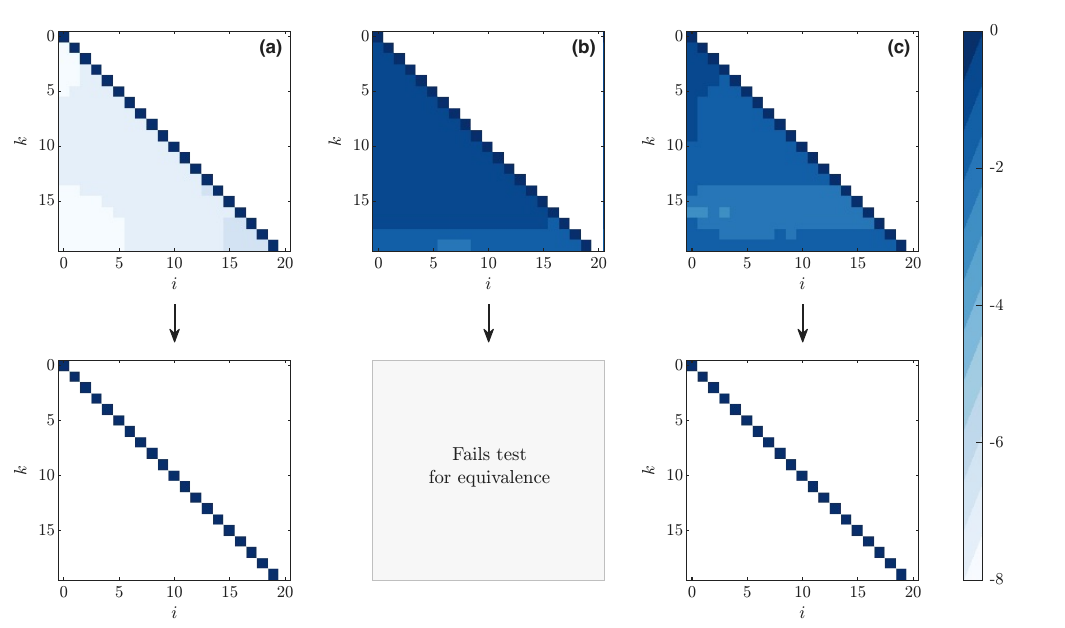}
	\caption{Heatmaps of $\log_{10}\|K_{k,i}\|_F$ (top row) and the corresponding $K_M$ (bottom row) for (a) the nondegenerate constrained case, (b) the case where $L$ is restricted to be block diagonal, and (c) the singular output-feedback case. 
	In (b), the constructed gain shows large residuals in the test for policy equivalence, as shown in \cref{tab:markov-diagnostics}. Hence, $\pi^M$ constructed using \cref{alg:markov-construction} should not be used, as it is not equivalent to $\pi^C$ and can violate constraints.
	In contrast, although $K$ in (c) contains significant off-diagonal blocks, \cref{alg:markov-construction} recovers an equivalent current-estimate-affine policy in the sense of \cref{prop:output-fb}.}
	\label{fig:gain-heatmaps}
\end{figure*}

\section{Numerical Demonstration} \label{sec:numerical}

For the numerical study, we use the Youla parameterization for convex reformulation \cref{eq:K-L}. The problems are implemented in YALMIP \cite{lofbergYALMIPToolboxModeling2004} and solved with MOSEK \cite{mosekapsMOSEKOptimizationToolbox2025}\footnote{Code available at \url{https://github.com/naoya-kumagai/markov-cs-matlab}.}.

For each solution, let $K_{\mathrm{off}}$ contain the off-diagonal blocks of $K$. $K_M$ is recovered using \cref{alg:markov-construction} and \cref{eq:K_M}. We report
\begin{align}
\delta_{\mathrm{off}}
&:=\frac{\|K_{\mathrm{off}}\|_F}{\|K\|_F}, \nonumber\\
\delta_{\mathrm{cond}}
&:=\max_k\frac{\|P_{u,k}-P_{ux,k}P_{x,k}^{\dagger}P_{ux,k}^{\top}\|_F}
{\|P_{u,k}\|_F}, \nonumber\\
\delta_{\mathrm{supp}}
&:=\frac{\|(K-K_M)P_X^{1/2}\|_F}
{\|KP_X^{1/2}\|_F}. \nonumber
\end{align}
$\delta_{\mathrm{off}}$ measures the magnitude of the off-diagonal blocks of $K$; $\delta_{\mathrm{cond}}$ evaluates the validity of the equality in \cref{eq:P_e-equality}; $\delta_{\mathrm{supp}}$ evaluates the validity of the equality in \cref{lem:equivalence-stochastic-support}. All are normalized.
Here, $x_k$ denotes the state estimate in the output-feedback cases. The first metric measures the magnitude of the off-diagonal blocks of $K$; the latter two measure the equivalence of the implemented state-history-affine and Markovian controls. By \cref{thm:markov,thm:nondegenerate}, near-zero $\delta_{\mathrm{off}}$ implies near-zero $\delta_{\mathrm{cond}}$ and $\delta_{\mathrm{supp}}$, but the converse is not true.
Note that the numerators in the metrics have different units: for $\delta_{\mathrm{off}}$, the units are those of the gain matrix; for $\delta_{\mathrm{cond}}$, they are those of the squared control; and for $\delta_{\mathrm{supp}}$, they are those of the control. Hence, $\epsilon$-optimality of the convex solver may yield different orders of magnitude for each metric, even if the Markovian equivalence is exact.

\subsection{Problem Setup}
The \emph{nondegenerate constrained} case uses the double-integrator problem with perfect state knowledge from \cite{okamotoOptimalCovarianceControl2018}, with $N=20$, $P_0=\operatorname{diag}(0.1,0.1,0.01,0.01)$, $G_k=0.01I_4$, and $P_f=0.5P_0$. Two state chance constraints are imposed using $(a_1,b_1)=([0.2,-1,0,0]^\top,0.2)$ and $(a_2,b_2)=([0.2,1,0,0]^\top,0.2)$, each with violation probability $5\times10^{-4}$. Here, $P_X\succ0$ so \cref{thm:nondegenerate} (and \cref{thm:markov}) apply.

The \emph{restricted-$L$} case uses the same fully observed dynamics, distributions, cost, and state chance constraints as the nondegenerate constrained case. The only difference is that the convex variable $L$ is constrained to be block diagonal. The argument in \cref{sec:limiting-sol-space} applies.

The \emph{singular output-feedback} case uses the same dynamics and constraints, with
\begin{equation*}
C_k=\begin{bmatrix}0_{3\times1}&I_3\end{bmatrix},\qquad D_k=10^{-3}I_3.
\end{equation*}
For this case only, the initial covariance is divided between the prior estimate and error covariances in proportions $0.75$ and $0.25$, respectively. Each four-dimensional estimate update is driven by a three-dimensional innovation, so the stacked estimate covariance is singular. Hence, $P_{\widehat{X}}$ is singular and we cannot expect $K$ to be block-diagonal as in \cref{thm:nondegenerate}, while \cref{prop:output-fb} holds.

\begin{table}[H]
\centering
\caption{Markov-policy diagnostics for the three numerical cases.}
\label{tab:markov-diagnostics}
\begin{tabular}{lccc}
\toprule
 & Nondegenerate & Restricted $L$ & Singular output \\
\midrule
$\delta_{\mathrm{off}}$ & $3.57\times10^{-7}$ & $4.00\times10^{-1}$ & $8.05\times10^{-2}$ \\
$\delta_{\mathrm{cond}}$ & $1.19\times10^{-12}$ & $5.98\times10^{-1}$ & $1.34\times10^{-9}$ \\
$\delta_{\mathrm{supp}}$ & $2.39\times10^{-7}$ & $2.60\times10^{-1}$ & $8.58\times10^{-6}$ \\
\bottomrule
\end{tabular}
\end{table}

\subsection{Results}
The results are summarized in \cref{tab:markov-diagnostics} and \cref{fig:gain-heatmaps}.
In the nondegenerate case, all three residuals are small, consistent with $K=K_M$ in \cref{thm:nondegenerate}. Note that in this case, although the diagonal blocks of $K$ can be used directly to construct $H_k$, the formula \cref{eq:H-e-def} is used to compute $\delta_{\mathrm{cond}}$ and $\delta_{\mathrm{supp}}$ to enable a fair comparison with the other two cases.

In the restricted-$L$ case, all three metrics are large; limiting the convex variable $L$ produces a controller that has genuine history dependence which cannot be equivalently represented, at least via \cref{alg:markov-construction}. This result is consistent with the discussion in \cref{sec:limiting-sol-space}.

In the singular output-feedback case, the off-diagonal blocks of $K$ have nonnegligible magnitudes, but both $\delta_{\mathrm{cond}}$ and $\delta_{\mathrm{supp}}$ remain small. Thus, the results provide numerical evidence of almost-sure equivalence.

\subsection{Solver Tolerance Study}
To distinguish structural residuals from finite solver accuracy, \cref{fig:tolerance-markov-metrics} repeats all three cases with MOSEK convergence tolerances (primal/dual feasibility and relative gap) from $10^{-4}$ to $10^{-9}$. Both $\delta_{\mathrm{cond}}$ and $\delta_{\mathrm{supp}}$ decrease by several orders of magnitude as the tolerance is tightened in the nondegenerate and singular output-feedback cases. In contrast, the restricted-$L$ residuals remain nearly constant, with $\delta_{\mathrm{cond}}$ and $\delta_{\mathrm{supp}}$ on the order of $10^{-1}$, confirming that they arise from the policy restriction rather than solver tolerance.

\begin{figure}[!t]
	\centering
	\includegraphics[width=\linewidth]{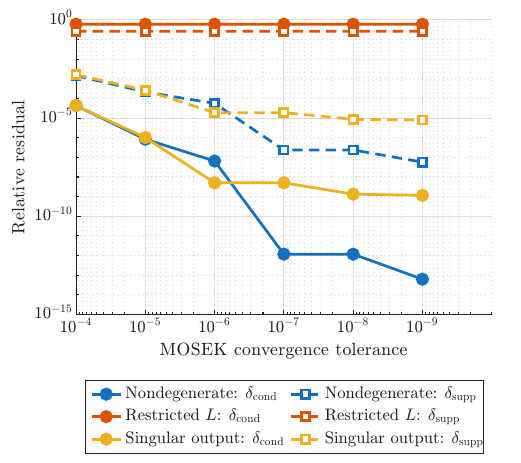}
	\caption{Markov-equivalence residuals versus the common MOSEK primal-feasibility, dual-feasibility, and relative-duality-gap convergence tolerance.}
	\label{fig:tolerance-markov-metrics}
\end{figure}


\section{Conclusions}
We investigate the optimality of deterministic Markovian policies for chance-constrained covariance steering.
We provide an algorithm that takes an optimal state-history-affine policy as input and constructs a Markovian policy that we prove to produce identical controls almost surely and thus the same state and control moments.
Put another way, the equivalence results show that Markovian policies are optimal within state-affine policies, for chance-constrained covariance steering problems where the constraints are covariance-monotonic. This extends the known result for unconstrained problems.  
We can now interpret the use of state-history-affine policies (even for chance-constrained problems) as a lifting of the Markovian policy space; at optimality, the lifted solution can be projected back losslessly.
For a horizon of length $N$, the constructed Markovian controller reduces online gain storage from $N(N+1)/2$ matrices to $N$ matrices and eliminates the need to retain the state history and its corresponding gains during execution. We extend the results to output-feedback scenarios and convex surrogates for value-at-risk.
We validate the policy equivalence numerically.

\crefalias{section}{appendix}
\useRomanappendicesfalse
\appendices

\section{Block Matrix Definitions} \label{sec:block-matrix-definitions}
\counterwithin{equation}{section} 
\setcounter{equation}{0}
\renewcommand{\theequation}{\thesection\arabic{equation}} 

For $0\leq i\leq j\leq N$, define the state-transition matrix
\begin{equation}
	\Phi_{j,i}:=
	\begin{cases}
		I, & j=i,\\
		A_{j-1}A_{j-2}\cdots A_i, & j>i.
	\end{cases}
\end{equation}
The lifted matrices in the batch dynamics are
\begin{equation}
	A:=
	\begin{bmatrix}
		\Phi_{0,0}^\top & \Phi_{1,0}^\top & \cdots & \Phi_{N,0}^\top
	\end{bmatrix}^\top
	\in\R^{n(N+1)\times n},
\end{equation}
and the $(j,i)$ blocks of $B$ and $G$ are
\begin{subequations}\label{eq:lifted-BG-definitions}
\begin{align}
	[B]_{j,i}
	&:=\begin{cases}
		\Phi_{j,i+1}B_i, & 0\leq i<j\leq N,\\
		0, & \text{otherwise},
	\end{cases}\\
	[G]_{j,i}
	&:=\begin{cases}
		\Phi_{j,i+1}G_i, & 0\leq i<j\leq N,\\
		0, & \text{otherwise}.
	\end{cases}
\end{align}
\end{subequations}
Thus, $B\in\R^{n(N+1)\times mN}$ and $G\in\R^{n(N+1)\times lN}$. Both are strictly block lower triangular.

Define the block down-shifted matrices by
\begin{subequations}
\begin{align}
	[Z_A]_{j,i}
	&:=\begin{cases}
		A_i, & j=i+1,\\
		0, & \text{otherwise},
	\end{cases}\\
	[Z_B]_{j,i}
	&:=\begin{cases}
		B_i, & j=i+1,\\
		0, & \text{otherwise}.
	\end{cases}
\end{align}
\end{subequations}
Here, $Z_A\in\R^{n(N+1)\times n(N+1)}$ and $Z_B\in\R^{n(N+1)\times mN}$. The open-loop response matrix
\begin{equation}\label{eq:F-definition}
	F:=(I-Z_A)^{-1}
\end{equation}
is block lower triangular, with
\begin{equation} \label{eq:F-blocks}
	[F]_{j,i}:=
	\begin{cases}
		\Phi_{j,i}, & j\geq i,\\
		0, & j<i.
	\end{cases}
\end{equation}
In particular, $F^{-1}=I-Z_A$ and $B=FZ_B$.

Finally, define
\begin{subequations}
\begin{align}
	\Sigma_w
	&:=\operatorname{blkdiag}
	(P_0,G_0G_0^\top,\dots,G_{N-1}G_{N-1}^\top),\\
	S
	&:=AP_0A^\top+GG^\top
	 =F\Sigma_wF^\top. \label{eq:S-definition}
\end{align}
\end{subequations}
Both $\Sigma_w$ and $S$ belong to $\R^{n(N+1)\times n(N+1)}$.
See the proof of \cref{lem:PD-equivalence} for the derivation of \cref{eq:F-blocks} and the second equality in \cref{eq:S-definition}.

\section{General Solution under Singular State Covariance} \label{sec:general-solution-singular-covariance}

If any state covariance matrix is singular, \cref{eq:H-e-def} in \cref{alg:markov-construction} is not a unique solution for $H_k$.
To see this, first, note a known result concerning pseudoinverses:
\begin{lemma}[Ch. 2, \cite{benIsraelGeneralizedInverses2003}] \label{lem:pseudoinv}
	For matrices $C \in \R^{p \times q}$ and $D \in \R^{r \times q}$, the following are equivalent:
	\begin{itemize}
		\item There exists $X \in \R^{r \times p}$ such that $XC = D$.
		\item $D = D C^{\dagger} C $
	\end{itemize}
	If these hold, the general solution for $X$ is
	\begin{equation} \label{eq:general-sol-pseudoinv}
		X = D C^\dagger + M (I - C C^\dagger)
	\end{equation}
	for arbitrary $M \in \R^{r \times p}$.
\end{lemma}

As may be evident from \cref{lem:pseudoinv}, \cref{eq:H-e-def} used in \cref{alg:markov-construction} is a specific solution to the equation $H_k P_{x,k} = P_{ux,k}$. We can obtain a general solution for $H_k$ as follows:
\begin{proposition}
	\cref{thm:markov} applies when $\pi^M$ is constructed with $H_k$ in \cref{alg:markov-construction} using the general form
	\begin{equation} \label{eq:H_k-general}
		H_k = P_{ux,k} P_{x,k}^{\dagger} + M (I - P_{x,k} P_{x,k}^{\dagger})
	\end{equation}
	with arbitrary $M \in \R^{m \times n}$.
\end{proposition}
\begin{proof}
	The only location where the specific definition of $H_k$ is used is in \cref{lem:ek-properties}.
	Observing the proof of (a) in \cref{lem:ek-properties}, 
	\begin{equation}
		P_{ux,k} = P_{ux,k} P_{x,k}^{\dagger} P_{x,k}
	\end{equation}
	is equivalent to \cref{eq:H_k-general}, from \cref{lem:pseudoinv}.
	The proof of (b) in \cref{lem:ek-properties} remains unchanged because, by \cref{lem:pseudoinv}, \cref{eq:H_k-general} is equivalent to $H_k P_{x,k} = P_{ux,k}$.
\end{proof}
Thus, if any state covariance matrix $P_{x,k}$ is rank-deficient, $H_k$ is not unique. However, in practice, there seems to be no need to choose a nonzero $M$ in \cref{eq:H_k-general}; that is, the solution in \cref{eq:H-e-def} can be used in every case. An analogous claim holds for the output-feedback case in \cref{prop:output-fb}.

\section*{Acknowledgment}
The authors acknowledge the use of Google Gemini and OpenAI's ChatGPT to assist in verifying the mathematical logic and streamlining the algebraic derivations of the proofs. All content was reviewed and edited by the authors, who take full responsibility for the final work.

\bibliographystyle{IEEEtran}
\bibliography{references}

\begin{IEEEbiography}[{\includegraphics[width=1in,height=1.25in,clip,keepaspectratio]{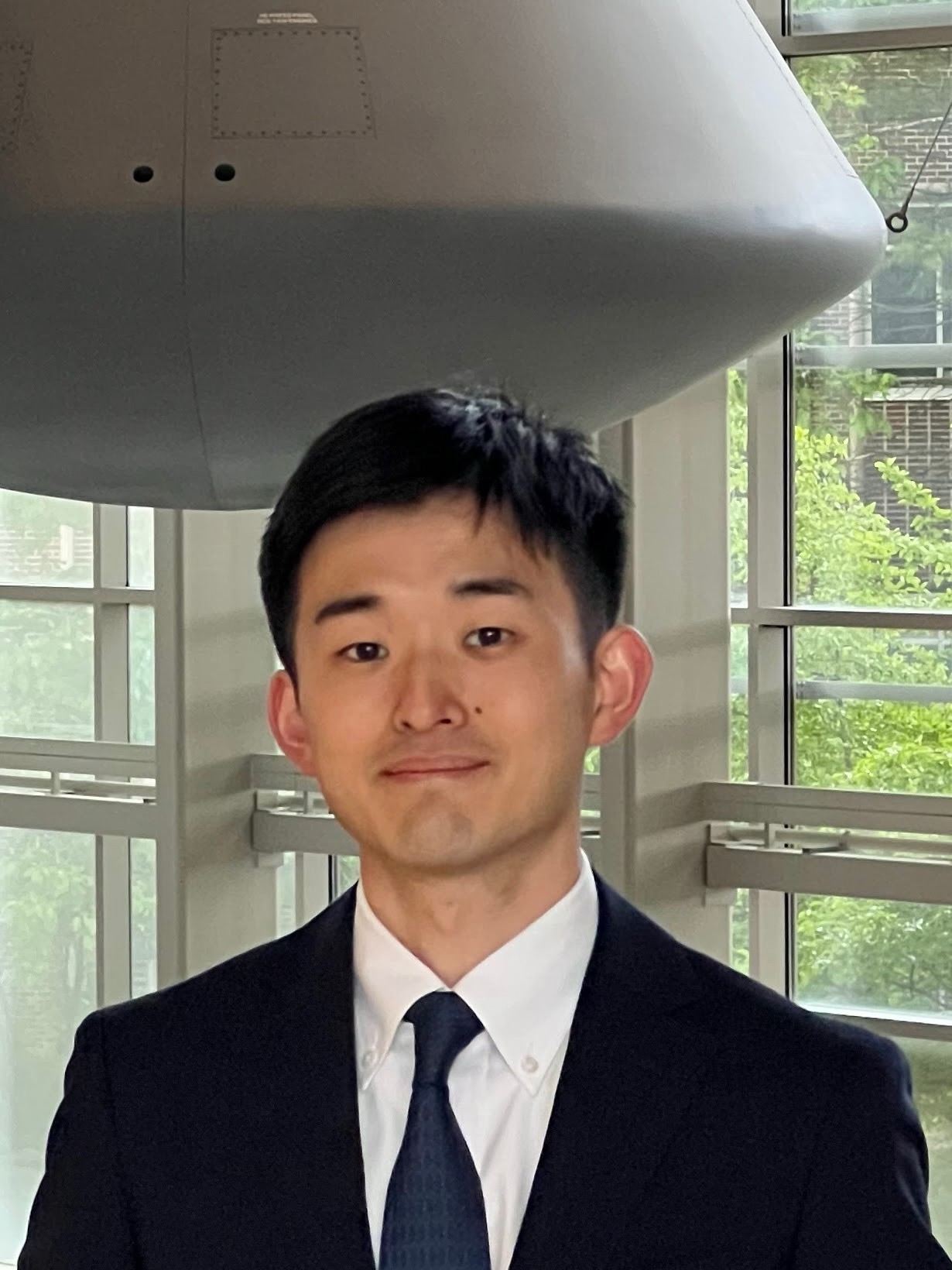}}]{Naoya Kumagai} (Graduate Student Member, IEEE) 
received the B.S. degree in Applied Mathematics from the University of California, Los Angeles, Los Angeles, CA, USA, in 2022. He is currently working toward the Ph.D. degree in aeronautics and astronautics with Purdue University, West Lafayette, IN, USA.

He has held internships in mission design and guidance, navigation, and control (GNC) at NASA JPL and ispace, inc.
His current research interests include stochastic control, optimization algorithms, and trajectory optimization.
\end{IEEEbiography}

\begin{IEEEbiography}[{\includegraphics[width=1in,height=1.25in,clip,keepaspectratio]{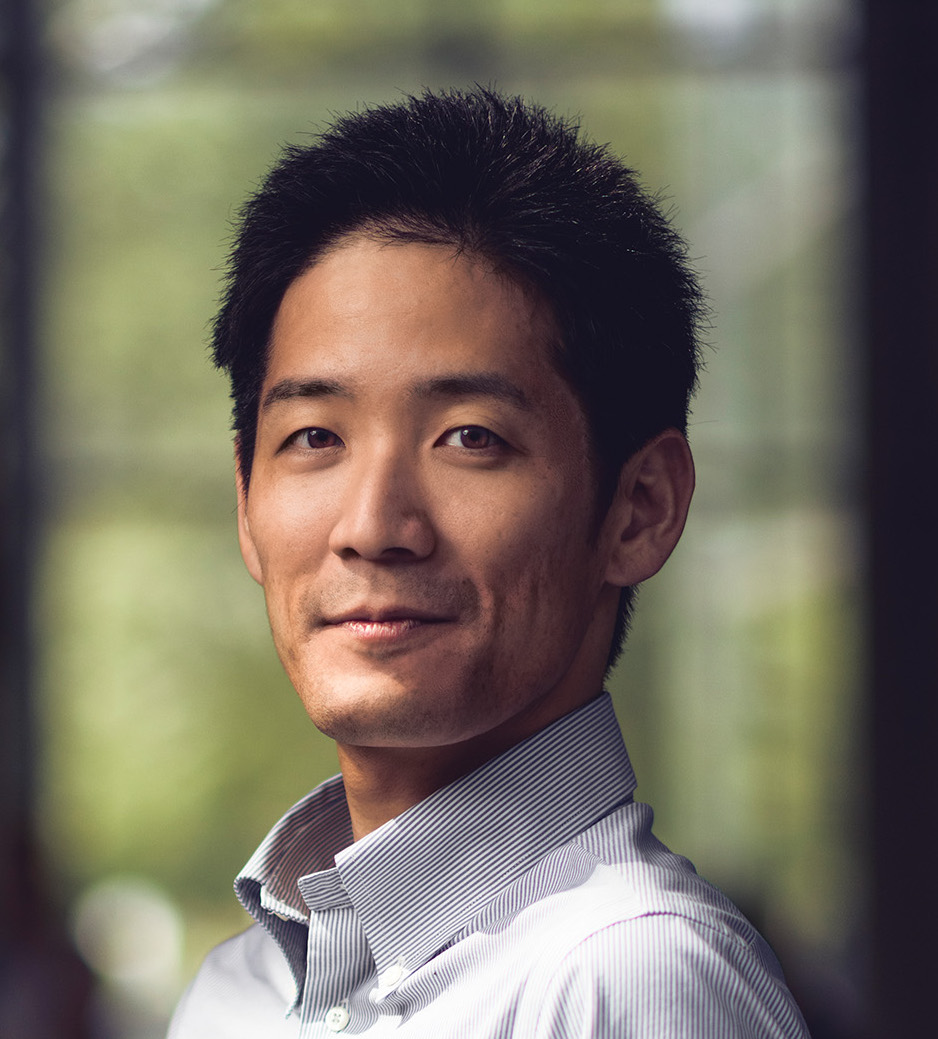}}]{Kenshiro Oguri} (Member, IEEE) is Assistant Professor of Aeronautics and Astronautics at Purdue University. He received his Ph.D. from the University of Colorado Boulder in 2021, and M.S. and B.S. from the University of Tokyo in 2017 and 2015, respectively. Prior to joining Purdue faculty in 2022, he worked for NASA JPL and JAXA. Ken's research interest includes orbital mechanics, control theory, dynamical systems, uncertainty quantification, and optimization. On the control-theoretic side, his research spans stochastic control, optimal control, and optimization. On the space application front, he develops analytical and numerical methods for space trajectory optimization, mission design, guidance navigation control (GNC), and autonomy. His research has been sponsored by NASA, JPL, AFOSR, Aerospace Corporation, and Draper Labs. Together with his students, he has published more than 120 journal and conference papers. His research has been recognized by NASA Early Career Faculty award and multiple paper awards, including AAS John V. Breakwell student paper award, AAS GNC conference best student paper award, and ACC best student paper award finalist, among others.
\end{IEEEbiography}

\end{document}